\documentclass{amsart}
\usepackage{pdfsync}

\usepackage{amsmath}%
\usepackage{amsthm}%
\usepackage{amsfonts}
\usepackage{amssymb}%
\usepackage{graphicx}
\usepackage{subfigure}
\newtheorem{thm}{Theorem}[section]
\theoremstyle{plain}

\newtheorem{lemma}[thm]{Lemma}
\newtheorem{prop}[thm]{Proposition}
\theoremstyle{remark}
\newtheorem{rem}[thm]{Remarks}
\newtheorem{remark}[thm]{Remark}

\newcommand{\Tone}{\mathbb{T}}

\newcommand{\trho}{\tilde{\rho}}  
\newcommand{\brho}{\bar{\rho}}

\newcommand{\R}{{\rm I \! R}}
\newcommand{\Cr}[1][r]{\mathcal{C}_{#1}}
\newcommand{\rx}{\rho_x}
\newcommand{\rxx}{\rho_{xx}}

\newcommand{\rt}{\rho_{\theta}}
\newcommand{\rtt}{\rho_{\theta \theta}}

\newcommand{\rxt}{\rho_{x \theta}}

\newcommand{\g}{\mathcal{G}}

\newcommand{\C}{\mathcal C}
\newcommand{\cH}{\mathcal H}

\numberwithin{equation}{section}
\begin{document}
\title[SD Flow of Perturbed Cylinders]{On the Flow of Non--Axisymmetric 
Perturbations of Cylinders via Surface Diffusion}

\author[J. LeCrone]{Jeremy LeCrone}
\address{Department of Mathematics \\ Kansas State University\\ Manhattan, KS USA}
\email{lecronjs@ksu.edu}%
\urladdr{http://www.math.ksu.edu/~lecronjs/}

\author[G. Simonett]{Gieri Simonett}
\address{Department of Mathematics \\ Vanderbilt University \\ Nashville, TN USA}
\email{gieri.simonett@vanderbilt.edu}%
\urladdr{http://www.vanderbilt.edu/math/people/simonett}



\begin{abstract}
We study the surface diffusion flow acting on a class of general (non--axisymmetric) perturbations of cylinders $\mathcal{C}_r$ in 
${\rm I \! R}^3$. Using tools from parabolic theory on uniformly regular manifolds, and maximal
regularity, we establish existence and uniqueness of solutions to surface diffusion flow starting
from (spatially--unbounded) surfaces defined over $\mathcal{C}_r$ via scalar height functions
which are uniformly bounded away from the central cylindrical axis. Additionally, we show that $\mathcal{C}_r$ is normally stable with respect to 
$2 \pi$--axially--periodic perturbations if the radius $r > 1$,and unstable if $0 < r < 1$. Stability is also
shown to hold in settings with axial Neumann boundary conditions.
\end{abstract}

\maketitle
\section{Introduction}

The surface diffusion flow is a geometric evolution law in which the normal velocity of 
a moving surface equals the Laplace--Beltrami operator of the 
mean curvature. More precisely, we assume in the following that
$\Gamma_0$ is a closed embedded surface in $\R^3$. Then the
surface diffusion flow is governed by the law
\begin{equation}
\label{SDF}
   V(t)=\Delta_{\Gamma(t)}H_{\Gamma(t)}\,,\qquad\Gamma(0)=\Gamma_0.
\end{equation}
Here $\Gamma=\{\Gamma(t) : t\ge 0\}$ is a family of  
hypersurfaces, $V(t)$ denotes the velocity in the normal
direction of $\Gamma$ at time $t$, while $\Delta_{\Gamma(t)}$ 
and $H_{\Gamma(t)}$ stand for the Laplace--Beltrami operator 
and the mean curvature of $\Gamma(t)$ 
(the sum of the principal curvatures in our case), respectively.
Both the normal velocity and the mean curvature depend on the 
local choice of the orientation.
Here we consider the case where $\Gamma(t)$ is embedded and 
encloses a region $\Omega(t)$, and we then choose the outward 
orientation, so that $V$ is positive if $\Omega(t)$ grows and 
$H_{\Gamma(t)}$ is positive if $\Gamma(t)$ is convex with respect to $\Omega(t)$.

The unknown quantity in \eqref{SDF} is the position and the geometry of
the surface $\Gamma(t)$, which evolves in time.
Hidden in the formulation of the evolution law \eqref{SDF} is
a nonlinear partial differential equation of fourth order.
This will become more apparent below, where the equation
is stated more explicitly.

It is an interesting and significant fact that the surface diffusion
flow evolves surfaces in such a way that the volume enclosed by $\Gamma(t)$
is preserved, while the surface area decreases 
(provided, of course, these quantities are finite).
This follows from the well-known relationships
\begin{equation}
\label{volume}
\frac{d}{dt}{\rm Vol(t)} = \int_{\Gamma(t)} V(t)\,d\sigma
= \int_{\Gamma(t)}\Delta_{\Gamma(t)}H_{\Gamma(t)}\,d\sigma = 0
\end{equation}
and 
\begin{equation}
\label{area}
\begin{aligned}
\frac{d}{dt}{\rm A(t)}
= \int_{\Gamma(t)} V(t)H_{\Gamma(t)}\,d\sigma
& =  \int_{\Gamma(t)}\left(\Delta_{\Gamma(t)}H_{\Gamma(t)}\right)H_{\Gamma(t)}\,d\sigma \\
& = -\int_{\Gamma(t)}\left|{\rm grad}_{\Gamma(t)}H_{\Gamma(t)}\right|^2_{\Gamma(t)}
\,d\sigma \le 0,
\end{aligned}
\end{equation}
where ${\rm Vol(t)}$ denotes the volume of $\Omega(t)$ and $A(t)$ 
the surface area of $\Gamma(t)$, respectively. 
Hence, the preferred ultimate states for the surface diffusion flow in 
the absence of geometric constraints are spheres.
It is also interesting to note that the surface diffusion flow can 
be viewed as the $H^{-1}$--gradient flow of the area functional,
a fact that was first observed in \cite{Fi00}. 
This particular structure has been exploited in 
\cite{May00,May01} for devising numerical simulations.

The mathematical equations modeling surface diffusion go back to a 
paper by Mullins \cite{MUL57} from the 1950s, who was in turn 
motivated by earlier work of Herring \cite{HER51}.
Since then, the surface diffusion flow \eqref{SDF}
has received wide attention in the mathematical community
(and also by scientists in other fields), see for instance
Asai~\cite{As12},
Baras, Duchon, and Robert~\cite{BDR84},
Bernoff, Bertozzi, and Witelski~\cite{BBW98},
Cahn and Taylor~\cite{CT94}, 
Cahn, Elliott, and Novick--Cohen~\cite{CEN96},
Dav\'i and Gurtin~\cite{DG90},
Elliott and Garcke~\cite{EG97},
Escher, Simonett, and Mayer~\cite{EMS98},
Escher and Mucha~\cite{EMu10},
Koch and Lamm~\cite{KoLa12},
LeCrone and Simonett~\cite{LS13},
McCoy, Wheeler, and Williams~\cite{McCWW11}, 
and Wheeler~\cite{Wh11,Wh12}.
\medskip\\
The primary focus of this article is the surface diffusion flow \eqref{SDF}
starting from initial surfaces $\Gamma_0$ which are perturbations of 
an infinite cylinder 
\[
\Cr := \{(x,y,z) \in \R^3 : y^2 + z^2 = r^2, \; x \in \R \}
\] 
with radius $r > 0$, centered about the $x$--axis.
\medskip\\
The main results of this paper address
\begin{itemize}
\item[(a)] existence, uniqueness, and regularity of solutions for \eqref{SDF}
for initial surfaces $\Gamma_0$ that are (sufficiently smooth) 
perturbations of $\Cr$; 
\item[(b)] 
existence, uniqueness, and regularity of solutions for \eqref{SDF}
for initial surfaces $\Gamma_0$ that are (sufficiently smooth)
perturbations of ${\mathcal C}_{r}$ subject to periodic or 
Neumann--type boundary conditions in axial direction;
\item[(c)] 
nonlinear stability and instability results for $\Cr$ with respect to perturbations 
subject to periodic or Neumann--type boundary conditions in axial direction.
\end{itemize}
In section 2 we derive an explicit formulation of \eqref{SDF} 
for surfaces which are parameterized over $\Cr$ in normal direction by 
height functions $\rho: \Cr \to \R$, see \eqref{NASD}.
Equation \eqref{NASD} is a fourth--order, quasilinear, 
parabolic partial differential equation 
for which  well--posedness in the unbounded setting can be 
established by virtue of recent results for maximal regularity 
on uniformly regular Riemannian manifolds (c.f. \cite{SS14, Am13}). 
A proof of well--posedness is carried out in section 3 of the current article.
For other recent results regarding geometric evolution equations in 
unbounded settings, refer to Giga, Seki, Umeda \cite{GSU09, GSU11}, 
wherein the mean curvature flow is shown to close open ends of 
non--compact surfaces of revolution given appropriate decay rates
at the ends of the initial surface. The current setting differs 
from that of Giga et. al. in many ways; most notably, we consider 
surfaces which lack symmetry about the cylindrical axis; 
we control behavior at space infinity via $bc^{2,\alpha}$ regularity bounds 
(so surfaces do not become radially unbounded); and we consider the dynamics
of surfaces which are uniformly bounded away from the cylindrical axis. 

In \cite{LS13}, we studied the surface diffusion flow for the special case of 
\textit{axisymmetric} and $2\pi$--periodic perturbations $\Gamma_0$ of $\Cr$.
In this case it was shown that equilibria of \eqref{SDF} consist exactly of cylinders, 
$2k \pi$--periodic unduloids and nodoids.
In addition, we established nonlinear stability and instability 
results for cylinders as well as bifurcation results, 
where $r$ serves as a bifurcation parameter.
More specifically, it was shown that cylinders $\Cr$ are stable for $r>1$,
unstable for $r<1$, and that a subcritical bifurcation of equilibria occurs 
at $r=1$, with the bifurcating branch
consisting exactly of the family of $2\pi$--periodic unduloids.

In this paper we show that the stability and instability results of \cite{LS13} 
remain true for non--axisymmetric perturbations subject to periodic or 
Neumann--type boundary conditions in axial direction.
In particular, we show that small $2\pi$--periodic perturbations 
$\Gamma_0$ of a cylinder $\Cr$ with $r>1$ 
exist globally and converge exponentially fast to another nearby cylinder 
\[
C(\bar y,\bar z,\bar r):=\{(x,y,z)\in\R^3 : 
		(y-\bar y)^2+(z-\bar z)^2={\bar r}^2,\; x\in \R\}.
\]
It will be shown that the radius $\bar r$ is uniquely determined by the volume
of the solid enclosed by $\Gamma_0$ and bounded by the planes $x=0$ and $x=2\pi$.
In order to prove this result, we shall show that all 
equilibria $\Gamma$ of \eqref{SDF} which are $2\pi$--periodic in axial 
direction and sufficiently close to $\Cr$ (for $r>1$) are cylinders 
$C(\bar y,\bar z,\bar r)$, with $(\bar y,\bar z,\bar r)$ close to $(0,0,r)$.
This result, which is interesting by itself, 
is obtained by a center--manifold argument.
In order to prove the stability result, we demonstrate 
that every cylinder $\Cr$ with radius $r > 1$ is \textit{normally stable} 
with respect to $2 \pi$--periodic,
$bc^{2 + \alpha}$--regular perturbations. 
As discussed in section 5, these stability results also encompass 
perturbations satisfying Neumann--type boundary conditions 
in $x$-direction.

\section{Surface Diffusion Flow near Cylinders}
\label{Sec:SDFormula}

Throughout this section, take $r > 0$ fixed, and let $\Cr \subset \R^3$ 
be the (unbounded) cylinder of radius $r$ which is symmetric about the x--axis. 
Specifically,
\[
    \Cr := \{ \big( x, r \cos(\theta), r \sin(\theta) \big): x \in \R, \theta \in \Tone \},
\]
where $\Tone := [0, 2\pi]$ denotes the one--dimensional torus with 
$0$ and $2\pi$ identified and $\Tone$ is equipped with the \textit{periodic}
topology generated by the metric
\[ 
    d_{\Tone}(x,y) := \min \{|x - y|, 2 \pi - |x - y| \}, \qquad x, y \in \Tone.
\]
Further, we equip $\Cr$ with the Riemannian metric 
$g := dx^2 + r^2 \, d\theta^2$, inherited from $\R^3$ via embedding.

\begin{rem}
Two properties of $\Cr$ are easy to see immediately:
\smallskip\\
(a) $\Cr$ is a \textit{uniformly regular Riemannian manifold}, as
			defined by Amann \cite{Am14}.\\ 
			Indeed, this follows from 
			\cite[formula (3.3), Corollary 4.3, and Theorem 3.1]{Am14}, noting
			that $\Cr$ is realized as the Cartesian product of the circle $r S^1$ and $\R$.
			On a side line we note that the class of uniformly regular 
			(closed) Riemannian manifolds coincides with the class of closed, 
			complete manifolds of bounded geometry, see~\cite{DSS14}.
 \smallskip\\
(b) $\Cr$ is an equilibrium of \eqref{SDF} for every radius $r > 0$. \\
        Indeed, noting that the mean curvature $\mathcal{H}_{\Cr} \equiv 1/r$ is constant
        throughout the manifold, it follows that 
        \[
    	    \Delta_{\Cr} \mathcal{H}_{\Cr} \equiv 0,
        \]
        so that the normal velocity $V \equiv 0$ and the surface remains fixed in space.
\end{rem}

Now, consider a scalar-valued function $\rho : \Cr \rightarrow \R$ such that
\[
    \rho(p) > - r, \qquad \text{for all} \quad p \in \Cr,
\]
and define a new surface $\Gamma = \Gamma(\rho)$ \textit{over} the
reference manifold $\Cr$ as
\[
    \Gamma := \{ p + \rho(p) \nu(p): p \in \Cr \},
\]
where $\nu$ denotes the unit outer normal field over $\Cr$.
We also note that the points $p \in \Cr$ are uniquely determined
via the surface coordinates $(x, \theta) \in \R \times \Tone$, and so we also
view the height function $\rho$ as a function of the variables $x$ and $\theta$.
Thus, abusing notation, we will interchangeably refer to $\rho(p)$ and 
$\rho(x,\theta)$, where, in fact,
\[
    \rho(x,\theta) := \rho \big( (x, r \cos( \theta), r \sin( \theta) )\big), 
        \qquad (x, \theta) \in \R \times \Tone.
\]
Extending this slight abuse of notation, one easily identifies an explicit 
parametrization for $\Gamma(\rho)$ via 
\begin{equation}
\label{RadialParam}
    \varphi(x, \theta) := \big(x, r \cos(\theta), r \sin(\theta) \big) + 
                          \rho(x, \theta) \big(0, \cos(\theta), \sin(\theta) \big), 
                          \quad  (x, \theta) \in \R \times \Tone.
\end{equation}

With the above construction, we have a correspondence between height functions 
$\rho > - r$ and a particular class of embedded manifolds 
$\Gamma(\rho) \subset \R^3$ which we refer to as 
\textit{axially--definable} --- 
namely, those manifolds which admit an axial parametrization 
$\varphi > 0$ of the form \eqref{RadialParam}.
Within this class of manifolds, the geometry, time--dependent evolution, 
and regularity are determined by the height function $\rho$ itself.
We turn to the task of explicitly expressing these relationships, and deriving the 
governing equation for surface diffusion, in the following subsection.
For the purpose of expressing the necessary geometric quantities, we 
assume for the moment that $\rho = \rho(x, \theta)$ is smooth enough
for all following calculations to work.
The precise desired regularity of $\rho$ will be addressed in 
detail when we discuss existence and uniqueness of solutions.

\subsection{Geometry of $\Gamma(\rho)$}

Utilizing explicit parametrizations of the form \eqref{RadialParam}, 
we find the coefficients of the first fundamental form of $\Gamma(\rho)$ as
\[
    g_{11} = (1 + \rx^2), \qquad g_{12} = g_{21} = \rx\rt, 
			\qquad g_{22} = ((r + \rho)^2 + \rt^2),
\]
and the metric is thus given by 
$g = g_{11} \ dx^2 + 2 g_{12} \ dx d\theta + g_{22} \ d \theta^2$.
Then,
\[
    g^{11} = \frac{(r + \rho)^2 + \rt^2}{\g}, \qquad 
    g^{12} = g^{21} = \frac{-\rx \rt}{\g}, \qquad 
    g^{22} = \frac{1 + \rx^2}{\g},
\]
where the matrix $[g^{ij}]$ is the inverse of $[g_{ij}]$ and 
\begin{equation}\label{CalG}
    \g = \g(\rho) := \det (g_{ij}) = (r + \rho)^2 ( 1 + \rx^2) + \rt^2.
\end{equation}
Likewise, the second fundamental form is expressed via the coefficients
\begin{align*}
    \mathbb{I}_{11} &= \frac{(r + \rho) \rxx}{\sqrt{\g}},\\
    \mathbb{I}_{12} &= \mathbb{I}_{21} = \frac{(r + \rho) \rxt - \rx \rt}{\sqrt{\g}},\\
    \mathbb{I}_{22} &= \frac{(r + \rho)(\rtt - (r + \rho)) - 2\rt^2}{\sqrt{\g}}.
\end{align*}

The mean curvature 
$\mathcal{H}(\rho) := \mathcal{H}_{\Gamma(\rho)}$ and 
Laplace--Beltrami operator $\Delta_{\Gamma(\rho)}$ are 
expressed using well--known formulas, 
c.f. \cite{DOC92,PS13}.
From \cite[Equation (5)]{PS13},
\begin{equation}\label{MCGamma}
    \mathcal{H}(\rho) = \frac{ \rt^2 - (r + \rho) [ ( (r + \rho)^2  + \rt^2 )\rxx + 
        ( 1 + \rx^2 )\rtt - 2 \rx \rt \rxt ]}{ \g^{3/2} } \ + \ \frac{1}{\g^{1/2}} \ ,
\end{equation}
and, from \cite[Section 2.7]{PS13}, 
\begin{equation*}\label{DeltaGamma}
    \Delta_{\Gamma(\rho)} = \frac{1}{\sqrt{\g}} \left\{ \partial_x \left[ 
        \frac{(r + \rho)^2 + \rt^2}{\sqrt{\g}} \ \partial_x \, - \, 
        \frac{\rx \rt}{\sqrt{\g}} \ \partial_\theta \right]
      + \partial_\theta \left[ \frac{1 + \rx^2}{\sqrt{\g}} \ \partial_\theta \, - \, 
        \frac{\rx \rt}{\sqrt{\g}} \ \partial_x \right] \right\}.
\end{equation*}
The normal velocity of the surface $\Gamma(\rho) = \Gamma(\rho(t))$ is likewise
\begin{equation*}\label{NormalVelocityGamma}
    V_{\Gamma(\rho)} = \frac{(r + \rho) \rho_t}{\sqrt{\g(\rho)}} .
\end{equation*}

Throughout, we often streamline notation by omitting explicit reference to 
dependence upon the spatial or temporal variables; 
e.g. $\rho = \rho(t) = \rho(x, \theta) = \rho(t,p) = \rho(t,x,\theta)$
and $\Gamma(\rho) = \Gamma(\rho(t))$, etc.

\subsection{The Surface Diffusion Flow}

With the formulas from the last subsection, we can now express the governing equation
for surface diffusion of $\Gamma(\rho)$ as an evolution equation for the 
height function $\rho$ alone. 
Thus, defining the (formal) operator
\begin{equation}\label{EvolOperator}
\begin{split}
	G(\rho) := 
		\frac{1}{(r + \rho)} &\left\{ \partial_x \left[ \frac{(r + \rho)^2 + \rt^2}
				{\sqrt{\g(\rho)}} \ \partial_x \mathcal{H}(\rho) 
		- \frac{\rx \rt}{\sqrt{\g(\rho)}} \ 
				\partial_\theta \mathcal{H}(\rho) \right] \right.\\
	& \quad + \ \partial_\theta \left.\left[ 
		\frac{1 + \rx^2}{\sqrt{\g(\rho)}} \ \partial_\theta 	
			\mathcal{H}(\rho) 
		- \frac{\rx \rt}{\sqrt{\g(\rho)}} \ \partial_x \mathcal{H}(\rho) 
			\right] \right\} \ ,
\end{split}
\end{equation}
we arrive at the expression
\begin{equation}
\label{NASD}
\begin{cases}
    \rho_t(t,p) = [G(\rho(t))](p), & \text{for $t > 0, \, p \in \Cr$,}\\
		\rho(0) = \rho_0, & \text{on $\Cr$},
\end{cases}
\end{equation}
for the surface diffusion flow for axially--definable surfaces $\Gamma(\rho)$.

\begin{remark}
The notation $[G(\rho(t))](p)$ reflects the functional--analytic framework
we use to address equation \eqref{NASD}.
For each $t \ge 0$, we consider $\rho(t)$ as an element of an appropriate Banach
space of regular functions defined on $\Cr$.
Then, $G$ maps $\rho(t)$ to another function defined over $\Cr$
which we then evaluate pointwise using the notation $[G(\rho(t))](p)$. 
\end{remark}

\section{Existence and Uniqueness of Solutions}
\label{Sec:WellPosedness}

In this section we establish existence and uniqueness of solutions for \eqref{NASD}.
One essential tool that we use throughout is the property of maximal 
regularity, also called \emph{optimal regularity}. Maximal regularity
has received considerable attention in connection with nonlinear 
parabolic partial differential equations, 
c.f. \cite{Am95, An90, CS01, LUN95, PSZ09}.

\subsection{Maximal Regularity}\label{MaxReg}

Although maximal regularity can be developed in a 
more general setting, we will focus on the setting of 
\textit{continuous} maximal regularity and 
direct the interested reader to the references 
\cite{Am95, LUN95} for a general development of the theory.

Let $\mu \in (0,1], \, J := [0,T]$, for some $T > 0$, 
and let $E$ be a (real or complex) Banach space. 
Following the notation of \cite{CS01}, we define spaces 
of continuous functions on $\dot{J} := J \setminus \{ 0 \}$ 
with prescribed singularity at 0 as
\begin{equation}\label{Eqn:SingularContinuity}
\begin{aligned}
	&BU\!C_{1 - \mu}(J,E) := 
		\bigg\{ u \in C(\dot{J}, E): [t \mapsto t^{1 - \mu} u(t)] \in BU\!C(\dot{J},E) 
			\; \text{and} \\ 
	& \hspace{9em} \lim_{t \rightarrow 0^+} t^{1 - \mu} \| u(t) \|_E = 0 \bigg\}, 
		\quad \mu \in (0,1)\\
	&\| u \|_{B_{1 - \mu}} := \sup_{t \in J} t^{1 - \mu} \| u(t) \|_E,
\end{aligned}
\end{equation}
where $BU\!C$ denotes the space consisting of bounded, uniformly continuous functions.
We also define the subspace
\[
	BU\!C_{1 - \mu}^1(J,E) := \left\{ u \in C^1(\dot{J},E) : 
		u, \dot{u} \in BU\!C_{1 - \mu}(J, E) \right\}, \quad \mu \in (0,1)
\]
and we set
\[
BU\!C_0 (J,E) := BU\!C(J,E), \qquad BU\!C^1_0(J,E) := BU\!C^1(J,E).
\]
If $J = [0, a)$ for $a > 0$, 
then we set
\begin{align*}
	C_{1 - \mu}(J,E) &:= 
		\{ u \in C(\dot{J},E): u \in BU\!C_{1 - \mu}([0,T],E), \quad T < \sup J \},\\
	C^1_{1 - \mu}(J,E) &:= 
		\{ u \in C^1(\dot{J},E): u, \dot{u} \in C_{1 - \mu}(J,E) \}, \qquad \mu \in (0,1],
\end{align*}
which we equip with the natural Fr\'echet topologies induced by 
$BU\!C_{1 - \mu}([0,T],E)$ and $BU\!C_{1 - \mu}^1([0,T],E)$, respectively.

If $E_1$ and $E_0$ are a pair of Banach spaces such that 
$E_1$ is continuously embedded in $E_0$,
denoted $E_1 \hookrightarrow E_0$, we set
\begin{equation}\label{Eqn:MaxRegSpaces}
\begin{split}
	&\mathbb{E}_0(J) := BU\!C_{1 - \mu}(J,E_0), \qquad \mu \in (0,1],\\
	&\mathbb{E}_1(J) := BU\!C^1_{1 - \mu}(J,E_0) \cap BU\!C_{1 - \mu}(J,E_1),
\end{split}
\end{equation}
where $\mathbb{E}_1(J)$ is a Banach space with the norm 
\[
	\| u \|_{\mathbb{E}_1(J)} := \sup_{t \in \dot{J}} t^{1 - \mu} 
		\Big( \| \dot{u}(t) \|_{E_0} + \| u(t) \|_{E_1} \Big).
\]
It follows that the trace operator $\gamma: \mathbb{E}_1(J) \rightarrow E_0$, defined by 
$\gamma v := v(0)$, is well-defined and we denote by $\gamma \mathbb{E}_1$ the image of 
$\gamma$ in $E_0$, which is itself a Banach space when equipped with the norm
\[
	\| x \|_{\gamma \mathbb{E}_1} := \inf \Big\{ \| v \|_{\mathbb{E}_1(J)}: 
		v \in \mathbb{E}_1(J) \, \text{and} \, \gamma v = x \Big\}.
\]

Given $B \in \mathcal{L}(E_1, E_0)$, closed as an operator on $E_0$, 
we say $\big( \mathbb{E}_0(J), \mathbb{E}_1(J) \big)$ 
is a \emph{pair of maximal regularity} for $B$ and write 
$B \in \mathcal{MR}_{\mu}(E_1, E_0)$, if 
\[
	\left( \frac{d}{dt} + B, \, \gamma \right) \in 
		\mathcal{L}_{{\rm isom}}(\mathbb{E}_1(J), \mathbb{E}_0(J) 
		\times \gamma \mathbb{E}_1), \qquad \mu \in (0, 1),
\]
where $\mathcal{L}_{{\rm isom}}$ denotes the set of bounded linear isomorphisms. 
In particular, $B \in \mathcal{MR}_{\mu}(E_1, E_0)$ 
if and only if for every $(f, u_0) \in \mathbb{E}_0(J) \times \gamma \mathbb{E}_1$, 
there exists a unique solution $u \in \mathbb{E}_1(J)$ to the inhomogeneous Cauchy problem 
\[
\begin{cases}
	\dot{u}(t) + Bu(t) = f(t), &t \in \dot{J},\\
	u(0) = u_0.
\end{cases}
\]
Moreover, in the current setting, it follows that 
$\gamma \mathbb{E}_1 \, \dot{=} \, (E_0, E_1)_{\mu, \infty}^0$,
i.e. the trace space $\gamma \mathbb{E}_1$ is topologically 
equivalent to the noted continuous 
interpolation spaces of Da Prato and Grisvard, c.f. \cite{Am95, CS01, DPG79, LUN95}.

\medskip

We shall establish well--posedness of \eqref{NASD} in the setting
of \textit{little--H\"older} regular height functions $\rho$. 
These results seem to constitute the first existence and uniqueness results for the 
surface diffusion flow acting on unbounded (closed) surfaces.
Note that we enforce minimal conditions on the ends of the initial surface,
namely we look at surfaces which are uniformly bounded away from the 
axis of definition (i.e. the x--axis in our current setting) and
satisfy minimal regularity assumptions.

For the convenience of the reader, we include a brief definition
of the little--H\"older spaces on $\Cr$.
For $\alpha \in (0,1)$ and $U \subset \R^n$ open, we define 
$bc^{\alpha}(U)$ to be the closure of the bounded smooth functions $BC^{\infty}(U)$
in the topology of the bounded H\"older functions $BC^{\alpha}(U)$. 
Further, for $k \in \mathbb{N}$, we define $bc^{k + \alpha}(U)$
to be the space of $k$--times continuously differentiable functions
such that the $k^{\text{th}}$--order derivatives are in $bc^{\alpha}(U)$.
We then define the space of $(k + \alpha)$ little--H\"older regular 
functions on $\Cr$, $bc^{k + \alpha}(\Cr)$, via an atlas of local charts
and a subordinate localization system. For more details regarding function spaces 
on uniformly regular (and singular) Riemannian manifolds see \cite{Am13, SS14}. 

\subsection{Quasilinear Structure and Maximal Regularity}
Expanding terms of \eqref{EvolOperator}, it is straight--forward to see 
that $G(\rho)$ is a fourth--order quasilinear operator of the form
\[
  G(\rho) = -\mathcal{A}(\rho)\rho + F(\rho) :=
		- \left( \sum_{|\beta| = 3, 4} b_{\beta}(\rho, \partial^1\rho, \partial^2\rho)
				\ \partial^{\beta} \rho \right) 
    + F(\rho, \partial^1\rho, \partial^2\rho),
\]
where $\beta = (\beta_1, \beta_2) \in \mathbb{N}^2$ is a multi--index, 
$|\beta| := \beta_1 + \beta_2$ its length, 
$\partial^{\beta} = \partial^{(\beta_1, \beta_2)} 
:= \partial_x^{\beta_1} \partial_\theta^{\beta_2},$
and $\partial^k \rho$, for $k \in \mathbb{N}$,
denotes the collection of all derivatives 
$\partial^{\beta} \rho$ for $|\beta| = k$. 
This quasilinear structure plays an important role in our well--posedness 
results below, for which we must look more closely at the fine properties
of the variable coefficient linear operator $\mathcal{A}(\rho)$.

Expanding terms, we have the following highest--order variable 
coefficients for the linear operator $\mathcal{A}(\rho)$:
\begin{align*}
    b_{(4,0)}(\rho) &=  \ \frac{((r + \rho)^2 + \rt^2)^2}{\g^2},\\
    b_{(3,1)}(\rho) &= - \ \frac{4 \rho_x \rt ((r + \rho)^2 + \rt^2)}{\g^2},\\
    b_{(2,2)}(\rho) &=  \ \frac{2((r + \rho)^2 + \rt^2)(1 + \rx^2) + 4 \rx^2 \rt^2}{\g^2},\\
    b_{(1,3)}(\rho) &= - \ \frac{4 \rx \rt (1 + \rx^2)}{\g^2},\\
    b_{(0,4)}(\rho) &=  \ \frac{(1 + \rx^2)^2}{ \g^2}.
\end{align*}
Therefore, the principal symbol $\sigma [\mathcal{A}(\rho)]$ of the 
linear operator $\mathcal{A}$ satisfies
\begin{align}\label{SymbolOfA}
\sigma [\mathcal{A}(\rho)](p, \xi ) 
    &:= \sum_{j + k = 4} [b_{(j,k)}(\rho)](p) (i \xi_1)^j (i \xi_2)^k \nonumber\\
    &=  \sum_{j+k=4}[b_{(j,k)}(\rho)](p) \xi_1^j \xi_2^k \nonumber\\
    &= \frac{1}{\g^2} \bigg( ((r + \rho)^2 + \rt^2) \xi_1^2 + 
	 (1 + \rx^2) \xi_2^2 - 2 \rx \rt \xi_1 \xi_2 \bigg)^2\\
	&\ge  \frac{1}{\g^2} \bigg( (r + \rho)^2 \xi_1^2 + \xi_2^2 \bigg)^2\nonumber.
\end{align}

With these quantities explicitly expressed, we will show that
$\mathcal{A}$ satisfies the property of continuous maximal regularity on $\Cr$. 
Well--posedness of \eqref{NASD} then follows by exploiting the quasilinear 
structure of the parabolic equation.

\begin{prop}[maximal regularity]\label{UnboundedMaxReg}
Fix $\varepsilon > 0, \; \alpha \in (0,1)$, and take $\mu \in [1/2, 1]$ so that 
$4 \mu + \alpha \notin \mathbb{Z}$. 
Further, let
$V_{\mu} := bc^{4 \mu + \alpha}(\Cr) \cap [\rho > \varepsilon - r]$
be  the set of admissible height functions. Then  it follows that
\[
	(\mathcal{A}, F) \in C^{\omega} \bigg(V_{\mu}, 
		\mathcal{M}_\nu \big( bc^{4 + \alpha}(\Cr), 
			bc^\alpha(\Cr) \big) \times bc^{\alpha}(\Cr) \bigg),
\]
where $C^{\omega}$ denotes the space of real analytic mappings between 
(open subsets of) Banach spaces.
\end{prop}

\begin{proof}
The regularity of $\mathcal{A}(\cdot)$ and $F(\cdot)$ is a consequence of the
analyticity of the following mappings: (which is easy to show using
standard methods in nonlinear analysis and theory of function spaces)
\begin{align*}
	[\rho \mapsto 1/\rho] &: V_0 \rightarrow bc^{\alpha}(\Cr),\\
	[\rho \mapsto \partial^{\beta} \rho] &: bc^{\sigma + |\beta|}(\Cr) \rightarrow bc^{\sigma}(\Cr),\\
	[(\rho, h) \mapsto \rho h] &: bc^{\alpha}(\Cr) \times bc^{\alpha}(\Cr) \rightarrow bc^{\alpha}(\Cr), 
\end{align*}
where $V_0 := bc^\alpha(\Cr) \cap [\rho > \varepsilon - r]$.
Additionally, one notes that the regularity of 
$\mathcal{A} : V_{\mu} \rightarrow \mathcal{L}(bc^{4 + \alpha}(\Cr), bc^{\alpha}(\Cr))$
is determined by the regularity of the coefficients 
$b_{\beta}: V_{\mu} \rightarrow bc^{\alpha}(\Cr)$. 
Then, one is left only to show that $\mathcal{A}(\rho)$ is in the denoted
maximal regularity class $\mathcal{M}_{\nu}$ for $\rho \in V_{\mu}$.
Note that $\mathcal{M}_{\nu}(bc^{4 + \alpha}(\Cr), bc^\alpha(\Cr))$ is an open subset of 
$\mathcal{L}(bc^{4 + \alpha}(\Cr), bc^\alpha(\Cr))$
(by \cite[Lemma 2.5(a)]{CS01}, for instance); thus the regularity 
of $\mathcal{A}$ mapping into $\mathcal{M}_{\nu}(bc^{4 + \alpha}(\Cr), bc^\alpha(\Cr)$ is
determined by the regularity of $\mathcal{A}$ mapping into 
$\mathcal{L}(bc^{4 + \alpha}(\Cr), bc^\alpha(\Cr))$.

To show $\mathcal{A}(\rho) \in \mathcal{M}_{\nu}(bc^{4 + \alpha}(\Cr), bc^\alpha(\Cr))$,
by \cite[Theorem 3.6]{SS14}, it suffices to show that $\mathcal{A}(\rho)$
is uniformly strongly elliptic on $\Cr$; i.e. we need to show that there exist
constants $0 < c_1 < c_2$ such that
\[
    c_1 \le Re \big( \sigma [ \mathcal{A}(\rho)](p , \xi ) \big) \le c_2, 
	\quad \text{for } (p, \xi) \in \Cr \times \R^2, 
	\quad \text{with } |\xi| = 1 \; .
\]
However, these bounds are obvious from the expression \eqref{SymbolOfA}
and the assumption that admissible functions $\rho$ are uniformly bounded
from below and above on $\Cr$.
\end{proof}

\subsection{Well--Posedness of \eqref{NASD}}
We now prove existence and uniqueness of solutions to \eqref{NASD}. 
Due to the non--compact setting, we must take care of the behavior of 
the height function as the axial variable approaches $\pm \infty$. 
For this purpose, we concentrate on surfaces which remain uniformly
bounded away from the x--axis (i.e. height function $\rho > -r + \varepsilon$
for $\rho : \Cr \rightarrow \R$ and $\varepsilon > 0$).
Also note that functions in the spaces $bc^{\sigma}(\Cr)$ are 
bounded from above by definition. 

\begin{prop}[existence and uniqueness]\label{UnboundedLocalSolutions}
Fix $\varepsilon > 0, \; \alpha \in (0,1),$ and take
$\mu \in [1/2, 1]$ so that $4\mu + \alpha \notin \mathbb{Z}$.
For each initial value 
\[
    \rho_0 \in V_\mu := bc^{4 \mu + \alpha}(\Cr) \cap [\rho > \varepsilon - r],
\]
there exists a unique maximal solution to \eqref{NASD}
\[
    \rho(\cdot, \rho_0) \in C^1_{1 - \mu} \big(J(\rho_0), bc^{\alpha}(\Cr) \big) \cap 
        C_{1 - \mu} \big(J(\rho_0), bc^{4 + \alpha}(\Cr) \big),
\]
where $J(\rho_0) := [0, t^+(\rho_0)) \subset \R_+$ denotes the 
maximal interval of existence for initial data $\rho_0$.
Further, it follows that 
\[
    \mathcal{D} := \bigcup_{\rho_0 \in V_\mu} J(\rho_0) \times \{ \rho_0 \}
\]
is open in $\R_+ \times V_\mu$ and the map
$[(t, \rho_0) \mapsto \rho(t, \rho_0)]$ defines an analytic semiflow on $V_{\mu}$.
Moreover, if 
the solution $\rho(\cdot,\rho_0)$ satisfies:
\begin{itemize}
		\vspace{2mm}
		\item[{\rm (i)}] $\rho(\cdot,\rho_0) \in UC(J(\rho_0),bc^{4\mu+\alpha}(\Cr))$, and
		\vspace{2mm}
		\item[{\rm (ii)}] there exists $\delta > 0$ so that 
			${\rm dist}_{bc^{4\mu+\alpha}(\Cr)}(\rho(t,\rho_0), \partial V_{\mu}) > \delta$
			for all $t \in J(\rho_0)$.
\end{itemize}
Then it must hold that $t^+(\rho_0) = \infty$ and so $\rho(\cdot,\rho_0)$ is a 
global solution of $\eqref{NASD}$.
\end{prop}

\begin{remark}
Condition (ii) of the proposition can also be interpreted 
by explicitly identifying the boundary of the admissible set $V_{\mu}$. In particular, 
$\rho(\cdot,\rho_0)$ remains bounded away from $\partial V_{\mu}$ if and only if
there exists some $0 < M < \infty$ such that, for all $t \in J(\rho_0)$ 
\begin{itemize}
	\vspace{2mm}
	\item[{\rm (ii.a)}] $\rho(t,\rho_0)(p) \ge 1/M + \varepsilon - r$ for all $p \in \Cr$, and
	\vspace{2mm}
	\item[{\rm (ii.b)}] $\| \rho(t,\rho_0) \|_{bc^{4 \mu + \alpha}(\Cr)} \le M$.	
\end{itemize}
\end{remark}

\begin{proof}
Existence of maximal solutions is proved in the same way as 
Proposition 2.2 of \cite{LS13}, whereas the claim for global solutions
differs slightly from Proposition 2.3 of \cite{LS13} due to the fact that we 
do not have compact embedding of little--H\"older spaces over the non--compact
manifold $\Cr$. In particular, well--posedness
follows from \cite[Theorems 3.1 and 4.1(c)]{CS01},
\cite[Proposition 2.2(c)]{SS14}, and 
Proposition~\ref{UnboundedMaxReg},
while the semiflow properties follow from \cite{CS01}
when $\mu < 1$ and from \cite{An90} in case $\mu = 1$.
\end{proof}

\subsection{Well--Posedness: Axially Periodic Surfaces}\label{SubsecPerWP}

In order to address the stability of cylinders under the flow of \eqref{NASD} ---
for which we will use explicit information regarding the spectrum of the 
linearization of $G$ --- we restrict our setting  to 
one within which spectral calculations are tractable.

For $a > 0$ fixed, define the axial shift operator 
\[
    T_a := \big[ x \mapsto x + a \big]: \R \rightarrow \R,
\]
which naturally acts on functions $\rho \in bc^{\sigma}(\Cr)$ as
\[
    [T_a \rho](x, \theta) = \rho(x + a, \theta).
\]
We define
\[
    bc^{\sigma}_{a,per}(\Cr) := \{ \rho \in bc^{\sigma}(\Cr): T_a \rho =\rho \},
\]
which we refer to as axially $a$--periodic little--H\"older functions on $\Cr$. 
It is a straight--forward exercise to see that $bc^{\sigma}_{a,per}(\Cr)$
forms a closed subspace of $bc^{\sigma}(\Cr)$, and is hence a Banach space.
We thus consider the properties of the surface diffusion evolution operator
$G$ as it acts on the axially periodic spaces.

\begin{prop}[$G$ preserves periodicity]
If $\rho \in bc^{4 + \alpha}_{a,per}(\Cr) \cap [\rho > -r]$, then
$G(\rho) \in bc^{\alpha}_{a,per}(\Cr)$.
\end{prop}

\begin{proof}
It suffices to show that $T_a$ commutes with the nonlinear operator
$G$, since $\rho$ $a$--periodic would then imply 
$T_a G(\rho) = G( T_a \rho) \equiv G(\rho)$, as desired.
The fact that $T_a$ indeed commutes with $G$ follows directly from the
commutativity of $T_a$ with the operations:
\[
    [\rho \mapsto 1/\rho], \quad 
    [\rho \mapsto \partial^{\beta} \rho], \quad \text{and} \quad
    [(\rho, h) \mapsto \rho h], 
\]
where, as before, $\partial^{\beta} = \partial^{(\beta_1, \beta_2)} 
= \partial_x^{\beta_1} \partial_\theta^{\beta_2}$.
\end{proof}

\begin{remark}
The fact that $G$ preserves periodicity can also be seen directly
from the geometric setting. 
In particular, recall that $G$ was constructed by modeling
the evolution of the surface $\Gamma(\rho)$ via \eqref{SDF},
which depends only upon the geometry of the surface $\Gamma(\rho)$ itself.
Thus, if $\rho$ is periodic, then $\Gamma(\rho)$, and all relevant
geometric quantities on $\Gamma(\rho)$ must also be periodic,
and hence the evolution cannot break this periodicity.
\end{remark}

The well--posedness results for \eqref{NASD} simplify slightly 
when restricted to periodic surfaces, owing to the fact that
the evolution is determined by the restriction to one interval
of periodicity. 
In essence, periodic surfaces are expressed entirely in
a compact setting.

\begin{prop}[periodic well--posedness]\label{PeriodicWellPosedness}
Fix $\alpha \in (0,1), \ a > 0$ and take 
$\mu \in [1/2, 1]$ so that $4\mu + \alpha \notin \mathbb{Z}$.
For each initial value 
\[
	\rho_0 \in V_{\mu, per} := bc^{4 \mu + \alpha}_{a,per}(\Cr) \cap [\rho > - r],
\]
there exists a unique maximal $a$--periodic solution to \eqref{NASD}
\[
	\rho(\cdot, \rho_0) \in C^1_{1 - \mu} \big(J(\rho_0), bc^{\alpha}_{a,per}(\Cr) \big) \cap 
		C_{1 - \mu} \big(J(\rho_0), bc^{4 + \alpha}_{a,per}(\Cr) \big).
\]
It follows that 
\[
    \mathcal{D} := \bigcup_{\rho_0 \in V_{\mu,per}} J(\rho_0) \times \{ \rho_0 \}
\]
is open in $\R_+ \times V_{\mu,per}$ and the map
$[(t, \rho_0) \mapsto \rho(t, \rho_0)]$ defines an analytic semiflow on $V_{\mu,per}$.
Moreover, if there exists $0 < M < \infty$ so that, for all $t \in J(\rho_0)$,
\begin{enumerate}
        \vspace{2mm}
	\item $\displaystyle \rho(t, \rho_0)(p) \ge 1/M - r \quad \text{for all $p \in \Cr$, and}$
        \vspace{2mm}
	\item $\displaystyle \| \rho(t, \rho_0) \|_{bc^{4 \mu + \alpha}_{a,per}(\Cr)} \le M$,
\end{enumerate}
then it must hold that $t^+(\rho_0) = \infty$ and $\rho(\cdot, \rho_0)$ is a
\textit{global solution}.
\end{prop}

\begin{proof}
By the assumption $\rho_0 > -r$ and periodicity, it is clear that  $\rho_0$ is uniformly
bounded away from the $x$-axis.
According to Theorem~\ref{UnboundedLocalSolutions}, the surface diffusion flow~\eqref{NASD}
admits a unique solution $\rho(\cdot)=\rho(\cdot,\rho_0)$ on a maximal interval of existence 
$J(\rho_0)$. It, thus, only remains to show that $\rho(t)$ is $a$-periodic for
each $t\in J(\rho_0)$.
Let $\rho_a(t):=T_a\rho(t)$ for $t\in J(\rho_0)$.
Then 
\begin{equation*}
\begin{aligned}
	&\partial_t \rho_a(t)=T_a\partial_t\rho(t)=T_a G(\rho(t))=G(T_a\rho(t)),
			\quad t\in J(\rho_0), \\
	&\rho_a(0)=T_a\rho(0)=T_a\rho_0=\rho_0,
\end{aligned}
\end{equation*}
showing that $\rho(\cdot)$ and $\rho_a(\cdot)$ both are solutions of \eqref{NASD} 
with the same initial value.
By uniqueness, $T_a\rho(t)=\rho(t)$, 
implying that $\rho(t)$ is $a$-periodic for all $t\in J(\rho_0)$.

The global well--posedness result follows from compactness of the embedding 
$bc^{4 + \alpha}_{a, per}(\Cr) \hookrightarrow bc^{\alpha}_{a,per}(\Cr)$
and \cite[Theorem 4.1(d)]{CS01}.
\end{proof}

For $\rho\in bc_{a,per}^\alpha(C_r)$ we define the area of 
$\Gamma(\rho)=\{p+\rho(p)\nu(p): p\in \Cr\}$
and the volume of the solid enclosed by $\Gamma(\rho)$, respectively, in a natural way
by just considering the part of $\Gamma(\rho)$ on an interval of periodicity in the $x$-direction.
More precisely, we set
\begin{equation}\label{restricted-cylinder}
\begin{split}
	&\C_{r,a} =\{(x,y,z)\in\R^3: y^2+z^2=r^2,\; x\in [0,a]\}, \\
	&\Gamma(\rho,a)=\{p+\rho(p)\nu(p): p\in \C_{r,a}\}. 
\end{split}
\end{equation}
Then we have the following result.

\begin{prop}
\label{conservation}
Fix $\alpha \in (0,1), a > 0$ and take 
$\mu \in [1/2, 1]$ so that $4\mu + \alpha \notin \mathbb{Z}$.
For each initial value
\[
	\rho_0 \in V_{\mu, per} := bc^{4 \mu + \alpha}_{a,per}(\Cr) \cap [\rho > - r],
\]
let $\rho=\rho(t,\rho_0)$ be the solution of \eqref{NASD}.
Then the surface diffusion flow 
preserves the volume of the region enclosed by $\Gamma(\rho(t),a)$ 
and bounded by the planes $x=0$ and $x=a$.
Moreover, the flow reduces the surface area of $\Gamma(\rho(t),a)$.
\end{prop} 

\begin{proof}
Although the assertions follow from \eqref{volume}-\eqref{area}
it will be instructive to give an independent proof that 
only relies on basic computations.
A short moment of reflection shows that volume is given by
\begin{equation}\label{volume-periodic}
{\rm Vol}(\rho(t))=\frac{1}{2}\int_{[0,a]\times [0,2\pi]} (r+\rho(t))^2\, d\theta dx. 
\end{equation}
Using  the short form $\rho=\rho(t,x,\theta)$, 
and keeping in mind that $\rho (t,x,\theta)$
is periodic in $(x,\theta)\in [0,a]\times [0,2\pi]$, 
one immediately obtains from \eqref{EvolOperator}
\begin{equation*}
\frac{d}{dt}Vol(\rho(t)) =\int (r+\rho)\rho_t\, d\theta dx = 0.
\end{equation*}
Next we note that the surface area of $\Gamma(\rho(t),a))$ is given by
\[
{\rm A}(\rho(t))=\int_{[0,a]\times [0,2\pi]} \sqrt{\g(\rho(t))}\, d\theta dx .
\]
It follows from integration by parts (using periodicity) 
and formulas \eqref{CalG}--\eqref{EvolOperator} 
\begin{equation*}
\label{area-periodic}
\begin{aligned}
\frac{d}{dt}A(\rho(t)) 
&=\int\left[\frac{(r+\rho)(1+\rho_x^2)}{\sqrt{\g}}
-\partial_x\left(\frac{(r+\rho)^2 \rho_x}{\sqrt{\g}}\right)
-\partial_\theta\left(\frac{\rho_\theta}{\sqrt{\g}}\right)\right]\rho_t\,d\theta dx \\
&=\int  (r+\rho)\cH \rho_t\, d\theta dx \\
&=-\int \frac{1}{\sqrt{\g}}\left[((r+\rho)^2+\rho_\theta^2)(\cH_x)^2
	- 2\rho_x\rho_\theta (\cH_x \cH_{\theta} + 
	(1+\rho_x^2)(\cH_{\theta})^2\right]d\theta dx \\
&\le -\int \frac{1}{\sqrt{\g}}\left[(r+\rho)^2(\cH_x)^2 
	+ (\cH_{\theta})^2\right]\,d\theta dx. 
\end{aligned}
\end{equation*}
Hence $A(\rho(t))$ is non-increasing and acts as a Lyapunov functional for
the evolution in the $a$--periodic setting. 
We also conclude that $a$--periodic equilibria of \eqref{NASD} 
correspond exactly to ($a$--periodic) surfaces of constant mean curvature.
\end{proof}

Before moving on to establish dynamic properties of solutions, 
we state the following characterization for axial
periodic height functions defined on $\Cr$. This characterization allows
explicit access to calculations involving the linearization and spectrum
of $DG(0)$ in the periodic setting.

\begin{prop}\label{FourierRep}
Given $\rho \in V_{\mu} := bc^{4 \mu + \alpha}_{a,per}(\Cr),$
it follows that $\rho$ satisfies the Fourier series representation
\[
    \rho(x, \theta) = \sum_{m,n \in \mathbb{Z}} \hat{\rho}(m,n) e^{2 \pi i m x / a} e^{i n \theta},
        \qquad (x, \theta) \in \R \times \Tone,
\]
where 
\[
    \hat{\rho}(m,n) := \frac{1}{2 \pi a}\int_{[0,a] \times \Tone} \rho(x, \theta) e^{- 2 \pi i m x / a} e^{- i n \theta}\,d\theta dx.
\]
\end{prop}

\begin{proof}
This follows from the realization of $\rho$ as a function over $\R \times \Tone$ ---
as discussed in Section~\ref{Sec:SDFormula} --- and classic Fourier analysis results
for periodic functions in higher dimensions (c.f. \cite{ST87}).
\end{proof}

\section{Nonlinear Stability and Instability of Cylinders Under Periodic Perturbations}

For the remainder of the paper we assume $a = 2 \pi$, 
and consider the behavior 
of the evolution equation \eqref{NASD} acting on $2 \pi$--axial--periodic
admissible height functions defined over $\Cr$
(i.e. $\rho_0 \in bc^{2 + \alpha}_{2\pi, per}(\Cr)$ and $\rho_0 > -r$).
Specifically, we are interested in the nonlinear dynamics of \eqref{NASD}
in a neighborhood of $\rho \equiv 0$ (i.e. near the cylinder $\Cr$ itself).
Thus, we begin in the next subsection by determining the spectral properties 
of the linearized evolution operator $DG(0)$.

\subsection{Linearization of $G$}

Utilizing the expression \eqref{EvolOperator}, one readily computes
the Fr\'echet derivative of $G$, which simplifies considerably
upon evaluation at the constant function $\rho \equiv 0$ to the operator
$DG(0) \in \mathcal{L}
\big( bc^{4+\alpha}_{2\pi, per}(\Cr), bc^{\alpha}_{2\pi, per}(\Cr) \big)$
given by
\begin{equation}\label{DGatRS}
	DG(0) h = - \left\{ \partial_x^2 + r^{-2} \partial_\theta^2 \right\} 
                    \left\{ \partial_x^2 + r^{-2} \partial_\theta^2 + r^{-2} \right\}h, 
                    \qquad h \in bc^{4 + \alpha}_{2\pi, per}(\Cr).
\end{equation}
This can be seen as follows, for instance.
Since we already know that $G$ is differentiable, 
the Fr\'echet derivative $DG(0)$ can be computed as
\begin{equation*}
\begin{aligned}
	DG(0)h = \left.\frac{d}{d\varepsilon}\right|_{\varepsilon=0} G(\varepsilon h)
     =\left[\left.\frac{d}{d\varepsilon}\right|_{\varepsilon=0}
			\Delta_{\Gamma(\varepsilon h)}\right]\cH(0)
      +\Delta_{\Gamma(0)}\cH^\prime(0)h.\quad 
\end{aligned}
\end{equation*}
Here one notes that $\Delta_{\Gamma(\epsilon h)} \cH(0) \equiv 0$  
(since $\cH(0) = \cH_{\Cr} \equiv 1/r$ is constant) while  
\[
\Delta_{\Gamma(0)} = \Delta_{\Cr} = \partial_x^2+r^{-2}\partial_\theta^2,
\] 
and $\cH^\prime (0) = -( \Delta_{\Cr} + r^{-2} )$ follow from 
\cite[formula (31)]{PS13} (noting that, by orientation convention, 
mean curvature has the opposite sign in \cite{PS13}).

\begin{lemma} 
\label{spectrum}
The spectrum of $DG(0)$ consists entirely of eigenvalues and is given by
\begin{equation}\label{SpectrumDG}
	\sigma(DG(0)) = \left\{ - \left(m^2 + r^{-2} n^2 \right) 
                  \left(m^2 + r^{-2} n^2 - r^{-2} \right): 
		  m, n \in \mathbb{Z} \right\}.
\end{equation}
Moreover, $0$ is a semi-simple eigenvalue of multiplicity 3 and
\[
    N(DG(0)) = {\rm span}\{1, \cos(\theta), \sin (\theta) \}.
\]
\end{lemma}

\begin{proof}
Due to compactness of the embedding   
$bc^{4+\alpha}_{2\pi, per}(\Cr)$ into $bc^{\alpha}_{2\pi, per}(\Cr)$ 
we know that the spectrum of $DG(0)$ consists only of eigenvalues.
Taking advantage of Proposition~\ref{FourierRep}, we conclude that
\begin{equation*}
	\sigma(DG(0)) = \left\{ - \left(m^2 + r^{-2} n^2 \right) 
                  \left(m^2 + r^{-2} n^2 - r^{-2} \right): 
		  m, n \in \mathbb{Z} \right\}.
\end{equation*}
Recalling the expression \eqref{DGatRS} for the linearized evolution
operator, $DG(0)$ is realized as a Fourier multiplier 
acting on $h \in bc^{4 + \alpha}_{2\pi, per}(\Cr)$ as
\[
    DG(0)h = \sum_{m,n \in \mathbb{Z}} 
			\hat{h}(m,n)[(m^2 + r^{-2} n^2)(m^2 + r^{-2} n^2 - r^{-2})] e^{i m x} e^{i n \theta}. 
\]
Locating the kernel of the linearization, if $DG(0)h = 0$,
then the fact that $r > 1$ necessarily implies that $\hat{h}(m,n) = 0$
whenever $m \ne 0$, and for $|n| > 1$ when $m = 0$.
Thus, the kernel of $DG(0)$ coincides with the 3--dimensional 
subspace spanned by the first order cylindrical harmonics 
$\{1,\cos(\theta),\sin(\theta)\}$.

Similarly, if $DG(0) h \in N(DG(0))$, it follows that $\hat{h}(m,n) = 0$
for all $m \ne 0$ and for $|n| > 1$ when $m = 0$. 
Hence, $N(DG(0)^2) = N(DG(0))$ and $0$ is thus a semi--simple
eigenvalue of $DG(0)$ with both geometric and algebraic multiplicity 
$3$ and eigenspace $N$.
\end{proof}

Note that $DG(0) \subset (-\infty, 0]$ when $r\ge 1$, 
which we will exploit in order to obtain stability.  
However, accounting for the presence of $0$ in the 
spectrum $DG(0)$ --- and acknowledging the fact that 
the family of cylinders $\Cr$ are obviously not isolated
equilibria ---
we turn to a center-manifold analysis for the stability analysis.

\subsection{Local Family of Cylinders}

Before we can apply the main result of \cite{PSZ09}, 
we must address the local nature of the family of 
equilibria surrounding the cylinder.
We show in this section that the family of 
equilibria surrounding $\rho \equiv 0$ is 
a smooth 3--dimensional manifold in $bc^{4 + \alpha}_{2\pi, per}(\Cr)$.
Specifically, we shall show that the local family of equilibria
coincides precisely with the manifold of two--dimensional cylinders 
near $\Cr$, for which we provide an explicit parametrization.

Denote by $\mathcal{M}_{cyl} \subset bc^{4 + \alpha}(\Cr)$
the family of height functions $\brho = \brho(\bar{y},\bar{z},\bar{r})$ 
such that 
\begin{equation*}
\Gamma(\brho)=C(\bar y,\bar z,\bar r):=\{(x,y,z)\in\R^3 : (y-\bar y)^2+(z-\bar z)^2={\bar r}^2,\; x\in \R\}
\end{equation*}
i.e. $\Gamma(\brho)$ is an infinite cylinder with axis of symmetry
$(x, \bar{y}, \bar{z})$ and $\bar{r}>0$,
(in contrast to $\Cr$ which is centered about the axis $(x, 0, 0)$ with
radius $r$). It follows that $\brho = \brho(x,\theta)$ must be independent
of $x$, and, as depicted in Figure 1, it is clear that 
$\brho \in \mathcal{M}_{cyl}$ must satisfy the equation 
\begin{equation*}
((r + \brho) \cos(\theta)-\bar y)^2 + 
	((r+\brho) \sin (\theta)-\bar z)^2=\bar r^2.
\end{equation*}
\vspace{-2em}
\begin{figure}[h]
\centering
\includegraphics[height=6cm,width=9cm,clip=true,trim=1.3in 4.5in 1in 2.75in]{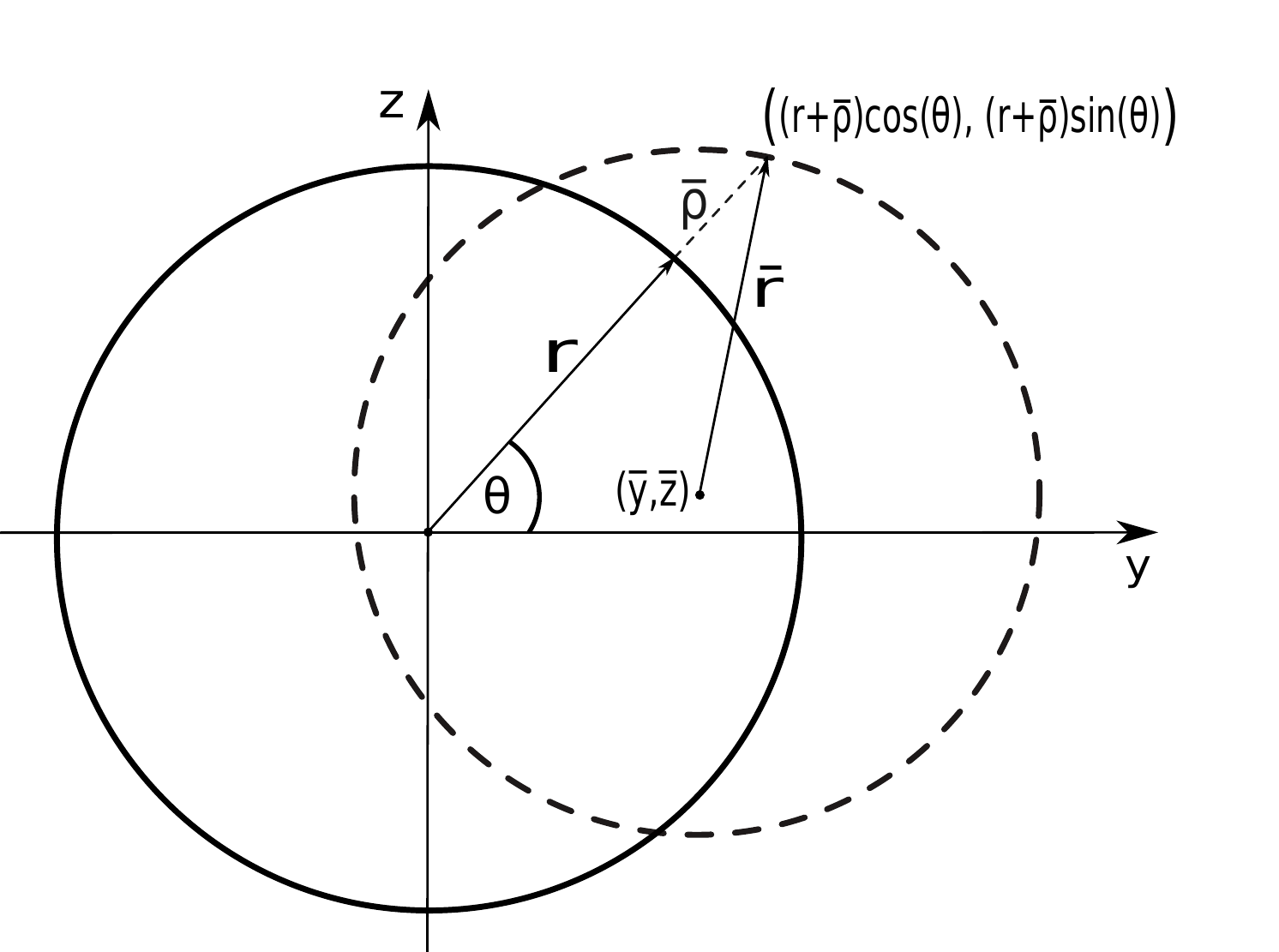}
\caption{Cross section of $C(\bar{y},\bar{z},\bar{r})$ defined over $\Cr$ via
	the height function $\brho = \brho(x,\theta)$.}
\end{figure}

\noindent Solving this quadratic equation for $(r + \brho)$ readily yields
the formula
\[
    \brho(x,\theta) = \cos(\theta)\bar{y} + \sin(\theta)\bar{z}
        + \sqrt{\bar{r}^2 - (\sin(\theta)\bar{y} - \cos(\theta)\bar{z})^2} - r.
\] 
The mapping $[(\bar{y},\bar{z},\bar{r}) \mapsto \brho]$ is a 
local diffeomorphism from some neighborhood
$\mathcal{O} \subset \R^3$ of $(0,0,r)$ into $bc^{4 + \alpha}_{2 \pi, per}(\Cr)$,
so that $\mathcal{M}_{cyl}$ is a 3--dimensional Banach manifold of 
equilibria containing $\rho \equiv 0$. 

Next, we wish show that $\mathcal{M}_{cyl}$ contains {\bf all} equilibria
of \eqref{NASD} in a neighborhood of $\rho \equiv 0$.
We accomplish this by showing that $\mathcal{M}_{cyl}$ coincides
locally with a stable center manifold  for \eqref{NASD}.

\begin{thm}\label{EquilMfd}
Given $r > 1$, \eqref{NASD} admits a unique, 3--dimensional stable 
center manifold in a neighborhood of $\rho \equiv 0$.
Moreover, this center manifold must coincide with $\mathcal{M}_{cyl}$,
which thus contains all equilibria near $\rho \equiv 0$.
\end{thm}
\begin{proof}
For 
$N = {\rm span}\{1, \cos(\theta), \sin (\theta) \}$
let
\ $\pi^c : bc^{\sigma}(\Cr) \rightarrow N$ be the orthogonal projection, given by
\[
		\pi^c h := 
			(h | 1) 1 + (h |\cos (\theta)) \cos(\theta) + (h |\sin(\theta)) \sin(\theta),
\]
where the inner product $(\cdot | \cdot)$ is defined for
$2 \pi$--periodic, $L_2$ integrable functions on $\Cr$.
In particular, 
\[
		(h | g) := \int_{\Tone^2} h(x,\theta) g(x,\theta) dx d\theta
\]
compares functions only over one interval of periodicity along the 
x--axis. The space of periodic $L_2$ functions over $\Cr$ clearly contains
$bc^{\sigma}_{2\pi,per}(\Cr)$ as a subspace, for any $\sigma > 0$.
Defining $\pi^s := id - \pi^c$ and $S := \pi^s(bc^{4 + \alpha}_{2\pi, per}(\Cr))$
we get the direct topological decomposition 
$bc^{4 + \alpha}_{2\pi, per}(\Cr)=N\oplus S$
which reduces the operator $DG(0)$,
as follows from the fact that 
$\pi^c$ commutes with $DG(0)$.
We conclude that $\pi^c$ coincides with the spectral projection of $DG(0)$
corresponding to the spectral set~$\{0\}$.
The existence of a (local) center manifold now
follows from \cite[Theorem 4.1]{Sim95}.
In particular, there exists a neighborhood $\mathcal{U}$
of $0 \in N$ and a mapping
\[
		\psi \in C^k(\mathcal{U}, bc^{4 + \alpha}_{2\pi, per}(\Cr)), 
				\qquad \text{for any } k\in \mathbb{N} \setminus \{0\},
\]
such that $\psi(0) = 0$, $\partial \psi (0) = 0$ and the graph
\[
		\mathcal{M}^c := 
			{\rm graph}(\psi) \subset N \oplus S = bc^{4 + \alpha}_{2 \pi, per}(\Cr)
\]
is a locally invariant manifold for the semiflow generated by
\eqref{NASD}, containing all global small solutions
(and, in particular, all equilibria close to $\rho = 0$). 
Moreover, it follows that the tangent space $T_0(\mathcal{M}^c)$
coincides with the eigenspace $N$, so $\mathcal{M}^c$ is a 
3--dimensional manifold containing $\rho=0$.

As $\mathcal M^c$ contains all equilibria that are close to $0$,
we know that $\mathcal M_{cyl}\subset \mathcal M^c$.
Since $\mathcal M_{cyl}$ and $\mathcal M^c$ both are 
(smooth) manifolds of the same dimension, 
a standard inverse function theorem argument 
yields $\mathcal{M}_{cyl} = \mathcal{M}^c$
in a sufficiently small neighborhood of $0$.
But this implies that every equilibrium close to 
$\rho=0$ lies on $\mathcal{M}_{cyl}$,
and thus is a cylinder.
\end{proof}

\subsection{Stability}

We are now prepared to prove the main result of this section 
regarding stability and instability of the family of cylinders $\Cr$ 
under sufficiently small admissible perturbations exhibiting
$2 \pi$--axial--periodicity.

\begin{thm}[stability and instability of cylinders]
\label{Stabil/Instabil}
Let $\alpha \in (0,1)$ and take $\mu \in [1/2, 1]$ so that
$4 \mu + \alpha \notin \mathbb{Z}$. 
\begin{enumerate}
\item (Exponential stability)
Let $r>1$ and $\omega\in (0, 1-r^{-2})$ be fixed.
Then there exist positive constants $\delta=\delta(\omega)$ and $M=M(\omega)$ such that, given an admissible periodic
perturbation
\[
	\rho_0 \in V_{\mu, per} := bc^{4 \mu + \alpha}_{2\pi, per}(\Cr) \cap [\rho > - r],
\]
with $\| \rho_0 \|_{bc^{4 \mu + \alpha}} < \delta$, the solution
$\rho(\cdot, \rho_0)$ exists globally in time, and there exists
a $\brho \in \mathcal{M}_{cyl}$ such that
\[
	\| \rho(t, \rho_0) - \brho \|_{bc^{4 \mu + \alpha}} 
		\le Me^{- \omega t} \|\rho_0\|_{bc^{4 \mu + \alpha}}.	
\]
Moreover, the radius $\bar{r}$ of the limiting cylinder $\brho$ is 
uniquely determined by preservation of periodic volume;
in particular, $\bar{r}$ is determined by the equality
$Vol(\rho_0) = Vol(\brho)$.

\item (instability)
Let $0<r<1$, then $\rho = 0$ is unstable in the topology of
$bc^{4\mu + \alpha}_{2\pi,per}(\Cr)$. 
\end{enumerate}
\end{thm}
\begin{proof}
(a) Lemma~\ref{spectrum} and Theorem~\ref{EquilMfd} show that
$\Cr$ is \emph{normally stable}, i.e. we have with $A=DG(0)$
\begin{itemize}
\item[{\rm (i)}] \,near $\rho=0$ the set $\mathcal E$ of equilibria constitutes a $C^1$-manifold 
in $bc^{4+\alpha}_{2\pi,per}(\Cr)$ of dimension~$3$,
\item[{\rm (ii)}] \, the tangent space for $\mathcal E$ at $0$ is given by $N(A)$,
\item[{\rm (iii)}] \, $0$ is a semi-simple eigenvalue of
                    $A$, i.e.\ $ N(A)\oplus R(A)=bc^\alpha_{2\pi,per}(\Cr)$,
\item[{\rm(iv)}] \, $\sigma(A)\setminus\{0\}\subset {\mathbb C}_{-}=\{z\in\C:\, {\rm Re}\, z<-\omega\}$ for some $\omega>0$
\end{itemize}
and exponential stability follows from \cite[Theorem 3.1]{PSZ09}.
To determine the radius $\bar{r}$ of the limit $\brho$, note that
Theorem~\ref{conservation} implies the equation 
${\rm Vol}(\brho) = {\rm Vol} (\rho_0)$ must be satisfied.
However, there is a one--to--one correspondence between the 
enclosed volume Vol$(\brho)$ (as defined by \eqref{volume-periodic})
and the radius $\bar{r}$ of $\brho$, hence $\bar{r}$ is uniquely
determined by volume preservation.
\smallskip\\
(b) Nonlinear instability of cylinders $\Cr$ with 
radius $0 < r < 1$ has already been shown in \cite[Theorem 5.1]{LS13}. 
To be precise, the authors proved 
this result in \cite{LS13} within the class of 
{\em axisymmetric} perturbations of the cylinder.
However, these perturbations also reside in the class of 
axially--definable surfaces, so nonlinear instability of 
the cylinder has already been established
in the current setting.
\end{proof}
\begin{rem}
(a) We remark that the proof of the  exponential stability result 
can also be based on a center manifold analysis, as in \cite[Theorem 1.2]{EMS98}. 
\smallskip\\
(b) 
There is a striking difference between the situation in 
\cite{EMS98} and the current situation concerning equilibria for the surface diffusion flow.
While it can readily be inferred that (all embedded, compact, closed) equilibria
of \eqref{SDF} are spheres (single spheres, or multiple spheres of the same radius),
the existence and classification of equilibria in the class of axially-definable surfaces 
is more difficult and complex. In the special case of 
\emph{axisymmetric and periodic} surfaces, equilibria are characterized 
by the Delaunay surfaces, see \cite{LS13}.
The general case for \emph{non--axisymmetric} surfaces seems to be wide open.
\smallskip\\
(c) Theorem 4.3 is interesting on its own, as it characterizes all 
\emph{axially-definable periodic} equilibria that are close to $\Cr$ for $r>1$. 
\smallskip\\
(d) We note that the radius of the cylinder $C(\bar y, \bar z, \bar r)$ to which
perturbations converge (when $r > 1$) is easily determined by the fact that 
surface diffusion flow preserves enclosed volume.
However, we cannot predict the location of the central axis 
$(\bar{y}, \bar{z})$.
\end{rem}

\section{Neumann-type boundary conditions}

It is also common to consider surfaces satisfying Neumann--type 
boundary conditions in the axial direction, c.f. \cite{ATH97, HAR12}.
We show in this section that the well--posedness and stability results
stated above continue to hold in subspaces of $bc^{\sigma}(\Cr)$ 
satisfying such boundary conditions.

Consider the restricted cylinder $\Cr[r,a]$ as in 
\eqref{restricted-cylinder} and define $bc^{\sigma}_{N}(\Cr[r,a])$
(for $\sigma > 3$) to be the collection of height functions
$\trho : \Cr[r,a] \to \R$ satisfying the boundary conditions
\[
		\partial_x \trho(0+, \theta) = \partial_x \trho(a-,\theta)=0 \quad\text{and}\quad 
		    \partial_x^3 \trho(0+,\theta) = \partial_x^3 \trho(a-,\theta)=0,
\]
for all $\theta \in \Tone$, and such that $\trho$ is $bc^{\sigma}$ 
regular up to the closed boundaries of $\Cr[r,a]$. 
It follows directly from \cite{Am13} and \cite{Am12} that
$bc^{\sigma}(\Cr[r,a])$ is a Banach space of
functions over the manifold with boundary $\Cr[r,a]$
It then follows that $bc^{\sigma}_N(\Cr[r,a])$ is a closed
subspace of $bc^{\sigma}(\Cr[r,a])$, and is hence a Banach space.

Before stating results for \eqref{NASD} in the setting of Neumann boundary
conditions, we first note that $bc^{\sigma}_N(\Cr[r,a])$ is linearly isomorphic 
to a closed subspace of $bc^{\sigma}_{2a,per}(\Cr)$ via the even extension operator
\[
  \rho(x, \theta) = \psi(\trho) (x, \theta) := 
  \begin{cases}  
      \trho(2k a - x, \theta) & 
          \text{if $x \in [(2k - 1) a, 2k a]$, $k \in \mathbb{Z}$}\\
      \trho(x - 2k a, \theta) &
          \text{if $x \in [2k a, (2k + 1)a]$, $k \in \mathbb{Z}$}.
  \end{cases}
\]
Thus, all results derived above in the axial periodic setting can also
be applied to surfaces satisfying Neumann
boundary conditions via restriction.

\begin{prop}\label{NeumannWellPosedness} (well--posedness for Neumann
boundary conditions)
Fix $\alpha \in (0,1), \ a > 0$ and take 
$\mu \in [3/4, 1]$ so that $4\mu + \alpha \notin \mathbb{Z}$.
For each initial value 
\[
	\trho_0 \in V_{\mu,N} := bc^{4 \mu + \alpha}_{N}(\Cr[r,a]) \cap [\rho > - r],
\]
there exists a unique maximal solution
\[
	\trho(\cdot, \trho_0) \in C^1_{1 - \mu} \big(J(\trho_0), bc^{\alpha}_{N}(\Cr[r,a]) \big) 
		\cap C_{1 - \mu} \big(J(\trho_0), bc^{4 + \alpha}_{N}(\Cr[r,a]) \big).
\]
to \eqref{NASD}. It follows that 
\[
    \mathcal{D} := \bigcup_{\trho_0 \in V_{\mu,N}} J(\trho_0) \times \{ \trho_0 \}
\]
is open in $\R_+ \times V_{\mu,N}$ and the map
$[(t, \trho_0) \mapsto \trho(t, \trho_0)]$ defines an analytic semiflow on $V_{\mu,N}$.
Moreover, if there exists $0 < M < \infty$ so that, for all $t \in J(\trho_0)$,
\begin{enumerate}
        \vspace{2mm}
	\item $\displaystyle \trho(t, \trho_0)(p) \ge 1/M - r \quad 
			\text{for all $p \in \Cr[r,a]$, and}$
        \vspace{2mm}
	\item $\displaystyle \| \trho(t, \trho_0) \|_{bc^{4 \mu + \alpha}_{N}(\Cr[r,a])} \le M$,
\end{enumerate}
then it must hold that $t^+(\trho_0) = \infty$ and $\trho(\cdot, \trho_0)$ is a
\textit{global solution}.
\end{prop}

\begin{proof}
Given $\trho_0 \in V_{\mu,N}$, the even extension $\rho_0 := \psi(\trho_0)$ 
is in the admissible class $V_{\mu, per}$ of $2a$--periodic functions on $\Cr$,
and hence Proposition~\ref{PeriodicWellPosedness} gives existence, uniqueness,
and global well--posedness of a periodic solution 
\[
	\rho(\cdot,\rho_0) \in C^1_{1 - \mu} \big(J(\rho_0), bc^{\alpha}_{2a,per}(\Cr) \big) 
		\cap C_{1 - \mu} \big(J(\rho_0), bc^{4 + \alpha}_{2a,per}(\Cr) \big)
\]
to \eqref{NASD} on $\Cr$. It remains to show that, upon restriction to $\Cr[r,a]$,
the solution $\rho(t,\rho_0)$ satisfies desired Neumann boundary conditions 
for all $t \in J(\rho_0)$.

Define the reflection operator $R$ about the $x=0$ plane so that
\[
	[R\rho](x,\theta) := \rho(-x,\theta).
\]
Then the restriction $\trho(t) := \rho(t) \big|_{\Cr[r,a]}$ satisfies the
desired boundary conditions if and only if 
\[
R\rho_x = \rho_x, \quad 
R\rho_{xxx} = \rho_{xxx}, \quad 
R(T_{-a}\rho_x) = T_{-a}\rho_x, \quad 
\text{and} \quad 
R(T_{-a}\rho_{xxx}) = T_{-a}\rho_{xxx} \ ,
\]
where $T_{-a}$ is the axial shift operator discussed in section~\ref{SubsecPerWP}.

From a geometric perspective, it is immediately clear that reflections about $x=0$
commute with the geometric calculations used to compute the evolution operator $G$,
however, confirming commutativity of $R$ and $G$ algebraically 
requires a little finesse.
In particular, note that $R$ commutes with the operations:
\[
    [\rho \mapsto 1/\rho], \quad 
    [\rho \mapsto \partial^k_{\theta} \rho], k \in \mathbb{N} \quad \text{and} \quad
    [(\rho, h) \mapsto \rho h],
\]
while 
\[
    R \, \partial_x^k = (-1)^k \partial_x^k R.
\]
However, one should observe that all odd axial derivatives of $\rho$
produced in the operator $G$ are found in products with an even number of 
similar terms. For instance, $\rho_x^2$,  $\rho_x \rho_{xxx}$, and $\rho_x \rho_{x \theta}$
are three such products found in the expansion of $G(\rho)$.
In this way, one confirms that the operators $R$ and $G$ indeed commute.

Finally, let $\rho_R(t) := R \rho(t)$ and $\rho_{R,a}(t) := R T_{-a} \rho(t)$ 
for $t\in J(\rho_0)$.
Then 
\begin{equation*}
\begin{aligned}
	&\partial_t \rho_R(t) = R \partial_t \rho(t) = 
            R G(\rho(t)) = G(R\rho(t)), \quad t \in J(\rho_0), \\
        &\rho_R(0) = R\rho(0) = R\rho_0 = \rho_0,
\end{aligned}
\end{equation*}
showing that $\rho(\cdot)$ and $\rho_R(\cdot)$ are both solutions of \eqref{NASD} 
with the same initial value.
By uniqueness, $R \rho(t) = \rho(t)$, 
implying that $\rho(t)$ satisfies Neumann boundary conditions at $x=0$
for all $t \in J(\rho_0)$.
Similarly, note that $\rho_{R,a}(t)$ and $T_{-a} \rho(t)$ are both solutions to 
\begin{equation*}
	\partial_t \rho(t) = G(\rho(t)), \; t \in J(\rho_0), 
	\qquad \rho(0) = T_{-a} \rho_0,
\end{equation*}
so that $\rho(t)$ also satisfies Neumann boundary conditions at $x=a$ and thus,
by restriction, we have a solution
\[
    \trho(t) := \rho(t) \big|_{\Cr[r,a]}
\]
satisfying Neumann boundary conditions on $\Cr[r,a],$ as desired.
\end{proof}

\begin{rem}
(a) Regarding statements of stability in the setting of Neumann boundary
conditions, 
we note that stability of cylinders was derived for $2\pi$--periodic 
perturbations of the unbounded cylinder $\Cr$.
Thus, by fixing $a = \pi$, one can derive stability and instability
of the bounded cylinder $\Cr[r,\pi]$ under perturbations satisfying
Neumann boundary conditions, with stability holding for $r > 1$
and instability for $0 < r < 1$ as before.
\smallskip\\
(b) We also note that the fundamental interval of periodicity has been taken
to be fixed at $2 \pi$ throughout the article, however this is only a 
convenience. As discussed in \cite[Remarks 6.2]{LeC14}, one can take 
perturbations of any periodicity $a > 0$, and derive analogous stability
of cylinders with radius $r > a / 2\pi$ and instability for cylinders
of radius $0 < r < a / 2\pi$. 
\end{rem}

\bigskip
{\bf Acknowledgment:}
The authors would like to thank Christoph Walker and Yuanzhen Shao for 
helpful discussions. This work was partially supported by
a grant from NSF (DMS-1265579 to G.~Simonett).
\bigskip
\bigskip

\end{document}